\documentclass[11pt,a4paper,reqno]{amsart}%[a4,leqno,11pt]{jsarticle}%{amsart}%{jsarticle}
\usepackage{amsmath,amsthm,amssymb,amscd}
\usepackage{graphicx}
\usepackage{ulem}
\graphicspath{{./figures/}}

\usepackage{ascmac}

 \usepackage{bm}
\usepackage[usenames]{color}
\usepackage{enumerate}
\usepackage{bigints}
\usepackage{amsthm}
\theoremstyle{plain}
 \newtheorem{theorem}{Theorem}[section]

 \newtheorem{proposition}[theorem]{Proposition}
 \newtheorem{fact}[theorem]{Fact}
 \newtheorem*{fact*}{Fact}
 \newtheorem{lemma}[theorem]{Lemma}
 \newtheorem{corollary}[theorem]{Corollary}
 \theoremstyle{remark}
 
 \newtheorem{remark}[theorem]{Remark}
 
 \newtheorem{example}[theorem]{Example}
\numberwithin{equation}{section}

\renewcommand{\Re}{{\rm Re}}
\renewcommand{\Im}{{\rm Im}}

\newcommand{\vect}[1]{\boldsymbol{#1}}
\newcommand{\inner}[2]{\left\langle{#1},{#2}\right\rangle}

\newcommand{\Nil}{\operatorname{Nil}_3}
\newcommand{\nil}{\mathfrak{nil}_3}

\usepackage{mathtools}
\mathtoolsset{showonlyrefs=true}%ref,eqrefしたものだけ番号が付く
\title[Framed null curves and timelike surfaces via Lorentzian harmonic maps]{Framed null curves and timelike surfaces via Lorentzian harmonic maps into de-Sitter 2-space}
\author[S.Akamine]{Shintaro Akamine}
\address[Shintaro Akamine]{
College of Bioresource Sciences,
Nihon University, 
1866 Kameino, Fujisawa, Kanagawa, 252-0880, Japan}
\email{akamine.shintaro@nihon-u.ac.jp}
\thanks{The first author was supported by JSPS KAKENHI Grant Number 23K12979.}

\author[H.Kiyohara]{Hirotaka Kiyohara}
\address[Hirotaka Kiyohara]{
Department of Mathematics Education, Osaka Kyoiku University, 4-698-1 Asahigaoka, Kashiwara, Osaka, 582-8582, Japan}
\email{kiyohara-h51@cc.osaka-kyoiku.ac.jp}
\thanks{The second author was supported by JST SPRING, Grant Number JPMJSP2119 and Foundation of Research Fellows, The Mathematical Society of Japan.}

\date{\today}

\keywords{Timelike constant mean curvature surface, Lorentzian harmonic map, Minkowski space, Heisenberg group}
\subjclass[2010]{Primary 53A10; Secondary 53B30, 57R45.}

\begin{document}
\maketitle

\begin{abstract}
We construct a class of Lorentzian harmonic maps into the de-Sitter $2$-space satisfying the eigenvalue equation $\Box N=2H^2N$ for the d'Alambert operator $\Box$ and a non-zero constant $H$ from framed null curves. We also investigate two classes of timelike surfaces associated with these Lorentzian harmonic maps: the first one is timelike surfaces with constant mean curvature $H$ in Lorentz-Minkowski $3$-space and the second one is timelike minimal surfaces in the three-dimensional Lorentzian Heisenberg group $\Nil(H)$. In particular, we characterize some properties of singularities on timelike minimal surfaces in $\Nil(H)$ via an invariant of framed null curves.
\end{abstract}

\section{Introduction}

The study of harmonic maps with values in space forms has developed together with the study of the corresponding classes of surfaces.
Beginning with the characterization by Ruh-Vilms \cite{RV}, which states that a surface in Euclidean $3$-space $\mathbb{E}^3$ has constant mean curvature (CMC) if and only if its Gauss map is a harmonic map into the unit sphere $\mathbb{S}^2$, Kenmotsu \cite{Kenmotsu} derived an explicit representation formula for CMC surfaces via harmonic maps.
More recently, it has been shown that such $\mathbb{S}^2$-valued harmonic maps also characterize spacelike maximal surfaces in the three dimensional Lorentzian Heisenberg group $\Nil$ (see \cite{BK,Lee}).

On the other hand, as a method for constructing harmonic maps, the DPW method \cite{DPW} based on the loop group method is now widely known, and this approach is applicable to harmonic maps into general symmetric spaces as well.
In particular, harmonic maps into the de-Sitter $2$-sphere $\mathbb{S}^2_1$, which are the main object of the present paper, have been studied from this viewpoint. See, for instance,~\cite{BS,DIT,KK}. 

Although such general methods for constructing harmonic maps are known, it remains difficult to construct harmonic maps and the corresponding surfaces with prescribed properties.
In this paper, we provide a systematic framework for analyzing a certain class of Lorentzian harmonic maps taking values in $\mathbb{S}^2_1$, as well as surfaces whose Gauss maps are given by such harmonic maps.

A Lorentzian harmonic map $\nu$ into $\mathbb{S}^2_1$ is characterized, via a suitable stereographic projection $\pi$ from $\mathbb{S}^2_1$ into the paracomplex plane $\mathbb{C}'$, by the harmonic map equation for $g := \pi \circ \nu$:
\begin{equation}\label{eq:Harmonic}
g_{z\bar{z}} - \frac{2\bar{g}}{1 + |g|^2} g_z g_{\bar{z}} = 0.
\end{equation}
Moreover, there is the following notable correspondence between surfaces whose Gauss maps are expressed by the same function $g$, which satisfies the equation \eqref{eq:Harmonic}.

\begin{fact}[\cite{AK1}, cf.  {\cite[Theorem 4.1]{KK}}]\label{fact:correspond}
For arbitrary real constant $H\neq 0$, any solution $g$ to the equation \eqref{eq:Harmonic} and the 1-form $\omega=\hat{\omega} dz$ defined by $\hat\omega = -j \bar{g}_{z} / (1+|g|^2)^2$, the following assertions hold.
\begin{itemize}
\item[(1)] The following surface $f_L$ with Gauss map $N_L$
\[
f_L = \frac{2}{H}\Re \bigintss \begin{pmatrix} g^2+1\\j(g^2-1)\\2jg \end{pmatrix}\omega,\quad N_L = \frac{ -1}{1+|g|^2} \begin{pmatrix} j (g-\bar{g}) \\ g+\bar{g} \\ 1-|g|^2 \end{pmatrix}, 
\]
gives a timelike constant mean curvature $H$ surface in Lorentz-Minkowski $3$-space $\mathbb{L}^3$ with signature $(-++)$.
\item[(2)]  The following surface $f$ 
\[
f =\begin{pmatrix} f^1\\f^2\\f^3 \end{pmatrix}=  \frac{2}{H}\Re \bigintss \begin{pmatrix} g^2+1\\j(g^2-1)\\2g-H f^2(g^2+1)+H f^1j(g^2-1) \end{pmatrix}\omega
\]
gives a timelike minimal surface in Lorentzian Heisenberg group $\Nil(H)$ with a left invariant metric $g_+$ whose Gauss map is  
\[
N=\frac{g+\bar{g} }{1-|g|^2} E_1+ \frac{j(g-\bar{g}) }{1-|g|^2} E_2+ \frac{1+|g|^2}{1-|g|^2} E_3,
 \]
 where $\{E_1,E_2,E_3\}$ is an orthonormal basis consisting of left invariant vector fields of $\Nil(H)$ so that $g_+(E_1, E_1)=-1$ and $g_+(E_2, E_2)=g_+(E_3, E_3)=1$.\end{itemize}
\end{fact}

For the Lorentzian Heisenberg group $(\Nil(H), g_+)$, see Section \ref{sec:2} for details. 
While there is a correspondence as in Fact \ref{fact:correspond}, explicitly constructing a solution of the harmonic map equation \eqref{eq:Harmonic} is not straightforward. 
Consequently, constructing the corresponding surfaces $f$ and $f_L$ with the prescribed properties is also difficult. 
In this paper, we therefore consider the construction of specific Lorentzian harmonic maps and the corresponding surfaces $f$ and $f_L$ from famed null curves. 

The main result of this paper is as follows.
\begin{theorem}\label{thm:Intro}
Let $H$ be a non-zero constant and $h=h(s)\colon I \to \mathbb{R}$ be a smooth function on an interval $I\subset \mathbb{R}$ with non-vanishing derivative. Then there exist 
\begin{itemize}
\item a null curve $\gamma=\gamma(s)$ with a null frame $(A,B,C)$ satisfying $d\gamma/ds = A$, 
\begin{equation}\label{eq:ABC_Intro}
A=\frac{S(h)}{H^2}B+\frac{1}{H^2}B'',\quad B= -\frac{H}{2h'} 
{ \left(-1-h^2, 1-h^2, 2h\right)},\quad C=\frac{1}{H}B'
\end{equation}
and the Frenet-Serret equations
\begin{equation}\label{eq:FS_Intro}
(A',B', C')=(A,B,C)\begin{pmatrix}
0 & 0 & H  \\
0 & 0  &  -S(h)/H \\
-S(h)/H  & H  & 0
\end{pmatrix},
\end{equation}
where $S(h)$ is the Schwarzian derivative of $h$,
\item a harmonic map $N_L=-C-tHB$ into $\mathbb{S}^2_1$ satisfying $\Box N_L=2H^2N_L$ with respect to the d'Alambert operator $\Box$ for the Lorentzian metric $\mathrm{I}=t^2H^2ds^2-2dsdt$,
\item a regular timelike B-scroll $f_L(s,t)=\gamma(s)+tB(s)$ in $\mathbb{L}^3$ with constant mean curvature $H$ whose first fundamental form is $\mathrm{I}$ and the Gauss map is $N_L$, and
\item a timelike minimal surface $f=f(s,t)$ in $\Nil(H)$ which corresponds to $f_L$ via the duality in Fact \ref{fact:correspond}.
\end{itemize}

Moreover, regarding on singularities of $f$ in $\Nil(H)$, the following assertions hold.
\begin{itemize}
\item[(1)] The surface $f$ always has singular points and the singular set is $\Sigma=\{(s,t)\in I\times \mathbb{R}\mid t=-\frac{C_3(s)}{HB_3(s)}\}$, where $B=(B_1,B_2,B_3)$ and $C=(C_1,C_2,C_3)$. 
\end{itemize}
\begin{itemize}
\item[(2)] The surface $f$ is a front at $(s,t)\in \Sigma$ if and only if $S(h)(s)\neq 0$.
\item[(3)] The surface $f$ is $\mathcal{A}$-equivalent to the cuspidal cross cap at $(s,t)\in \Sigma$ if and only if $S(h)(s)= 0$ and $S(h)'(s)\neq 0$.
\end{itemize}
\end{theorem}

Here, a singular point of $f$ is a point where $f$ is not immersed. For the definition of a front and for details on singular points, see Section \ref{sec:2}. A cuspidal cross cap appearing in $(3)$ above refers to a singular point of $f$ at which $f$ is locally diffeomorphic to $f_{CCR}(u,v)=(u,v^2,uv^3)$.

One observes that properties (2) and (3) in Theorem \ref{thm:Intro} are invariant under isometries of $\mathbb{L}^3$ acting on the null curve $\gamma$, as well as under linear fractional transformations of $h$ that leave $S(h)$ invariant.  Based on this observation, one can state the following invariance property of singularities. For more detailed statements, see Section 5. 

\begin{corollary}\label{cor:Intro}
For a function $h$ satisfying the condition (3) in Theorem \ref{thm:Intro}, there exists a family of timelike minimal surfaces $\{f^{O}\}_{O\in O(2,1)}$ in $\Nil(H)$ associated with the Lorentz orthogonal group $O(2,1)$ such that each surface $f^{O}$ admits a cuspidal cross cap.
\end{corollary}

Recently, spacelike maximal surfaces with singularities in $\Nil(H)$ have been investigated by Brander-Kobayashi \cite{BK}, and concrete examples of such surfaces admitting prescribed types of singularities have been provided. 
On the other hand, to the best of the authors' knowledge, even explicit examples of timelike minimal surfaces in $\Nil(H)$ possessing specific singularities such as cuspidal cross caps or swallowtails have not yet been known. 
In view of the above situation, starting with Corollary~\ref{cor:Intro}, we present in Sections 5 and 6 methods for constructing timelike minimal surfaces in $\Nil(H)$ with singularities such as cuspidal edges, swallowtails, and cuspidal cross caps, together with explicit examples.

The organization of this paper is as follows.  
In Section \ref{sec:2}, we review basic facts on singularities and theory on minimal surfaces in $\Nil(\tau)$ with singularities.  
In Section \ref{sec:3}, we describe the construction of timelike ruled surfaces in $\mathbb{L}^3$ with constant mean curvature $H$, called {\it B-scrolls}, from framed null curves in $\mathbb{L}^3$, and the corresponding timelike minimal surfaces in $\Nil(H)$ via Fact \ref{fact:correspond}.  
In particular, the part of the proof of Theorem \ref{thm:Intro} concerning the harmonic map $N_L$ is given.
In Section~\ref{sec:4}, the remaining part of the proof of Theorem \ref{thm:Intro} is given.  
First, in Theorem~\ref{thm:criteria_Bscroll}, we describe conditions for $f$ in Theorem \ref{thm:Intro} to be a front and to have a cuspidal cross cap singularity in terms of the curvature function $\kappa_2$, which is an invariant of framed null curves.  
Then, in Section~4.2, we show that any framed null curve is generated from a certain one-variable function $h$, and we state in Theorem \ref{thm:Schwarzian_criterion} that $\kappa_2$ is written by the Schwarzian derivative of $h$.  

In Section 5, as a consequence of Theorem~\ref{thm:Intro}, we observe in Corollary~\ref{cor:invariance1} that singularities on timelike minimal surfaces in $\Nil(H)$ corresponding to $B$-scrolls in $\mathbb{L}^3$ have certain invariance properties.  
In particular, Corollary~\ref{cor:Intro} is proved.  
We also discuss conditions under which $f$ has front singularities that are not cuspidal edges, such as swallowtails.

Finally, in Section 6, we present explicit examples of B-scroll type timelike minimal surfaces in $\Nil(H)$ with cuspidal cross caps and swallowtails.

%%%%%%%%%%%%%%%%%%%%%%%%%%%%%%%%%%
%%%%%%%%%%%%%%%%%%%%%%%%%%%%%%%%%%
%%%%%%%%%%%%%%%%%%%%%%%%%%%%%%%%%%
%%%%%%%%%%%%%%%%%%%%%%%%%%%%%%%%%%
\section{Minimal surfaces in $\Nil(\tau)$ with singularities}\label{sec:2}

In this section, we briefly review the necessary preliminaries on timelike minimal surfaces in Lorentzian Heisenberg group and singularities on them. For further details, see \cite{AK1}.

%%%%%
Let $\mathbb{C}'$ be the set of paracomplex numbers of the form $z=x+jy$, where $x,y\in \mathbb{R}$ and $j$ is the imaginary unit satisfying $j^2=1$. For each $z=x+jy\in \mathbb{C}'$, one can define the following notions:
\begin{itemize}
\item the real part $\Re{z}:=x$ and the imaginary part $\Im{z}:=y$,
\item the conjugate $\bar{z}:=x-jy$, and
\item squared modulus of $z$ as a $|z|^2:=z\bar{z}=x^2-y^2$.
\end{itemize}
It should be remark that the relation $|z|^2<0$ may hold in general and $|jz|^2=-|z|^2$ holds.
%%%%%

Let $\tau$ be a non-zero real number and $\Nil(\tau)$ be the Lorentzian Heisenberg group, which is identified with $\mathbb{R}^3$ endowed with the group multiplication
\begin{equation}\label{eq:group}
(x_1, x_2, x_3)\cdot (y_1, y_2, y_3) = \left(x_1+y_1, x_2+y_2, x_3+y_3+\tau(x_1y_2-x_2y_1)\right).
\end{equation}
In this paper, we consider the following left-invariant metric $g_+$ on $\Nil(\tau)$:
\begin{equation}\label{eq:g+}
g_+=-(dx_1)^2+(dx_2)^2+\eta(\tau)^2,\quad \eta(\tau)=dx_3+\tau \left(x_2dx_1-x_1dx_2\right).
\end{equation}
Formally setting $ \tau = 0 $, the space $(\mathrm{Nil}(\tau), g_+)$ becomes Lorentz-Minkowski $3$-space $\mathbb{L}^3$, and hence $g_+$ is one of the natural left-invariant metrics on $\mathrm{Nil}(\tau)$. For the metric $g_+$, the following left-invariant vector fields $\{E_1,E_2,E_3\}$ give a Lorentzian orthonormal vector fields
\[
E_1=\frac{\partial}{\partial x_1}-\tau x_2\frac{\partial}{\partial x_3},\quad E_2=\frac{\partial}{\partial x_2}+\tau x_1\frac{\partial}{\partial x_3},\quad E_3=\frac{\partial}{\partial x_3},
\]
which satisfy
\[
-g_+(E_1,E_1)=g_+(E_2,E_2)=g_+(E_3,E_3)=1,\quad g_+(E_k,E_l)=0\quad (k\neq l).
\]
Hence, the Lie algebra $\nil(\tau)$ of $\Nil(\tau)$ is naturally identified with $\mathbb{L}^3$ via the basis $e_1=(1,0,0)$, $e_2=(0,1,0)$ and $e_3=(0,0,1)$ satisfying the relations of the Lie bracket 
\[
[e_1,e_2]=2\tau e_3,\quad [e_1,e_3]=0,\quad [e_2,e_3]=0.
\]

An immersion $f\colon M \to \Nil(\tau)$ from a $2$-dimensional connected manifold into $\Nil(\tau)$ is called {\it timelike} if its induced metric $f^*g_+$ is a Lorentzian metric. Under the above identification of $\nil(\tau)$ with $\mathbb{L}^3$, the left translation $f^{-1}N$ of the  unit normal vector field $N$ of $f$ can be considered as a map valued in the de-Sitter $2$-space
\[
\mathbb{S}^2_1=\left \{\sum_{k=1}^3x_ke_k\in \nil(\tau)\mid -x_1^2+x_2^2+x_3^2=1\right \}.
\]
Let us consider the stereographic projection from $\mathbb{S}^2_1$ into $\mathbb{C}'$
\[
\pi \colon \sum_{k=1}^3{x_ke_k} \mapsto \frac{x_1}{1+x_3}+j\frac{x_2}{1+x_3}.
\]
Then the map $g:=\pi \circ f^{-1}N$ is called the {\it normal Gauss map} of $f$ and by using it $f^{-1}N$ can be written as
\[
f^{-1}N=\frac{2\Re{(g)}}{1-|g|^2}e_1+\frac{2\Im{(g)}}{1-|g|^2}e_2+\frac{1+|g|^2}{1-|g|^2}e_3.
\]
In general, it is known that a map
\[
\nu \colon D\subset \mathbb{C}' \to \mathbb{S}^2_1\subset \mathbb{L}^3_{-++}
\]
is harmonic if and only if $g:=-j(\pi_L \circ \nu)$ satisfies equation~\eqref{eq:Harmonic}; see \cite{AK1} and \cite{DIT}. 

\[
\pi_L \colon (x_1,x_1,x_3) \mapsto \frac{x_1}{1-x_3}+j\frac{x_2}{1-x_3}.
\]

There are two types of timelike surfaces whose unit normal vector fields are given by such harmonic maps taking values in $\mathbb{S}^2_1$: timelike constant mean curvature $\varepsilon H$ surfaces in $\mathbb{L}^3$ and timelike minimal surfaces in $\Nil(H)$, listed in Fact \ref{fact:correspond}.

\subsection{Singularities on fronts and frontals}\label{Sec.frontals}
Let $U$ be a domain in $\mathbb{R}^2$ and $u$, $v$ are local coordinates on $U$. 
A smooth map $f\colon U \longrightarrow N^3$ from $U$ into a Riemannian $3$-manifold $(N^3,g)$ is called a {\it frontal} if there exists a unit vector field $n$ along $f$ such that $n$ is perpendicular to $df(TU)$ with respect to the Riemannian metric $g$ of $N^3$. We call $n$ the {\it unit normal vector field} of $f$. Moreover if the map $L=(f,n)$ is an immersion, $f$ is called a {\it front}. 

A point $p\in U$ where a frontal $f$ is not an immersion is called a {\it singular point} of $f$, and we call the set of singular points of $f$ the {\it singular set}. 

Let $\Omega$ be the volume element of $(N^3,g)$. The function $\lambda=\Omega(f_u, f_v, n)$ on $(U;u,v)$ is called the {\it signed area density function} of the frontal $f$. Singular points are characterized as zeros of $\lambda$ and a singular point $p$ is called {\it non-degenerate} (resp.~{\it degenerate}) if $d\lambda_p \neq0$ (resp.~$d\lambda_p=0$). 

Let $U_i$ be domains of $\mathbb{R}^2$ and $p_i$ be points in $U_i$ ($i=1, 2$). Two smooth map germs $f_1\colon (U_1,p_1)\longrightarrow (N^3,f_1(p_1))$ and $f_2 \colon (U_2,p_2)\longrightarrow (N^3,f_2(p_2))$ are {\it $\mathcal{A}$-equivalent} if there exist diffeomorphism germs $\Phi\colon (\mathbb{R}^2,p_1) \longrightarrow (\mathbb{R}^2,p_2)$ and $\Psi\colon (N^3,f_1(p_1)) \longrightarrow (N^3,f_2(p_2))$ such that $f_2 = \Psi \circ f_1\circ \Phi^{-1}$. As typical singular points of frontals, a {\it cuspidal edge}, a {\it swallowtail} and a {\it cuspidal cross cap} are given as singularities which are $\mathcal{A}$-equivalent to 
\[
f_{CE}(u,v) =(u,v^2,v^3),\quad f_{SW}(u,v)=(3u^4+u^2v,4u^3+2uv,v),\quad f_{CCR}(u,v)=(u,v^2,uv^3)
\]
at origin, respectively. Cuspidal edges and swallowtails are singular points of fronts, and cuspidal cross caps are singular points of frontals but not fronts.

In this paper, we consider the situation of $N^3=\Nil(\tau)$.  
 The Heisenberg group $\Nil(\tau)$ is the space $\mathbb{R}^3$ equipped with the group multiplication \eqref{eq:group}, and one can consider the natural left-invariant Riemannian metric associated with \eqref{eq:g+},
\begin{equation}\label{eq:gR}
g_R = (dx^1)^2 + (dx^2)^2 + \eta(\tau)^2, \qquad 
\eta(\tau) = dx^3 + \tau \left( x^2 dx^1 - x^1 dx^2 \right).
\end{equation}
Thus, we may regard $(N^3, g)$ as $(\Nil(\tau), g_R) = (\mathbb{R}^3, g_R)$.

\subsection{Criteria for singularities via the duality between $f$ and $f_L$}
In this section, we recall the following criteria for singularities of timelike minimal surfaces in $\Nil(\tau)$ via the duality in Fact \ref{fact:correspond} proved in \cite{AK1}.
\begin{theorem}\label{thm:criteria2}
Let $z$ be a point of an extended timelike minimal surface $f$ in $\Nil(\tau)$, at which the corresponding constant mean curvature $\tau$ surface $f_L$ is immersed. Let $N_L$ be the unit normal vector field of $f_L$ in $\mathbb{L}^3$. Then, the following assertions hold.
\begin{itemize}
\item[(0)] $z$ is a singular point of $f$ if and only if $\langle N_L, \bm{e}_3\rangle =0$ at $z$,
\item[(1)] $z$ is a non-degenerate singular point of $f$ if and only if $\langle N_L, \bm{e}_3\rangle =0$ and $\langle dN_L, \bm{e}_3\rangle \neq 0$ at $z$.
\end{itemize}
From now on, we assume that $z$ is a non-degenerate singular point of $f$. Then, we can take a singular curve $c$ of $f$ passing through $c(0)=z$, and put $c_L:=f_L\circ c$. Then, 
\begin{itemize}
\item[(2)] $f$ is a front at $z$ if and only if $\langle c_L', \bm{e}_3\rangle  \neq 0$ at $z$,
\item[(3)] $f$ is $\mathcal{A}$-equivalent to the cuspidal edge at $z$ if and only if
\[
\langle c_L', \bm{e}_3\rangle  \neq 0 \quad \text{and}\quad c_L' \nparallel \bm{e}_3\quad \text{at $z$},
\]
that is $ c_L'$ is neither perpendicular nor parallel to $\bm{e}_3$ at $z$.
\item[(4)] $f$ is $\mathcal{A}$-equivalent to the swallowtail at $z$ if and only if
\[
\langle c_L', \bm{e}_3\rangle  \neq 0,\quad  c_L' \parallel \bm{e}_3\quad  \text{and} \quad  c_L'' \nparallel \bm{e}_3\quad \text{at $z$},
\]
\item[(5)] $f$ is $\mathcal{A}$-equivalent to the cuspidal cross cap at $z$ if and only if 
\[
\langle c_L', \bm{e}_3\rangle  = 0,\quad  c_L' \nparallel \bm{e}_3\quad  \text{and} \quad  \langle c_L'', \bm{e}_3\rangle  \neq  0 \quad \text{at $z$}.
\]
\end{itemize}
\end{theorem}

\begin{remark}
The above criteria in Theorem \ref{thm:criteria2} using the curve $c_L=f_L \circ c$ on $f_L$ does not depend on parametrizations of the singular curve $c$ and does not require the normal Gauss map $g$. Therefore, such criteria are useful when the constant mean curvature surface in $\mathbb{L}^3$ corresponding to a timelike minimal surface in $\Nil(\tau)$ is known concretely.
\end{remark}

%%%%%%%%%%%%%%%%%%%%%%%%%%%%%%%%%%
\section{B-scroll type surfaces in $\Nil(\tau)$}\label{sec:3}

Let us consider a null curve $\gamma=\gamma(s)$ defined on an interval $I\subset \mathbb{R}$ with a null frame $(A,B,C)$ satisfying
\begin{equation}\label{eq:frame_cond}
A:=\gamma':=\frac{d\gamma}{ds},\quad C:=A\times B,\quad \langle B,B\rangle=0,\quad \langle A,B\rangle=-1,
\end{equation}
and the Frenet-Serret equations
\begin{equation}\label{eq:FS}
(A',B', C')=(A,B,C)\begin{pmatrix}
\kappa_1 & 0 & \kappa_3 \\
0 & -\kappa_1  & \kappa_2 \\
\kappa_2 & \kappa_3  & 0
\end{pmatrix},
\end{equation}
where $\kappa_1=\langle A,B'\rangle$, $\kappa_2=\langle A',C\rangle$ and $\kappa_3=\langle B',C\rangle$.

The surface 
\begin{equation}\label{eq:Bscroll}
f_L(s,t)=\gamma(s)+tB(s),\quad (s,t)\in I\times \mathbb{R}
\end{equation}
 is called a {\it null scroll} associated with the frame $(A,B,C)$, and it is called a {\it B-scroll} when $\kappa_1=0$. The notion of B-scrolls was originally introduced by Graves \cite{Graves}. 
 
 A straightforward calculation shows that the null scroll $f_L$ in \eqref{eq:Bscroll} is a timelike surface and has the spacelike unit normal vector field $N_L=-C-t\kappa_3 B$. Here, we take the negatively oriented normal vector filed to consider the duality in Fact \ref{fact:correspond}, see also \cite[Remark 2.7]{AK1}. 
Moreover, $f_L$ has the following first and second fundamental forms:
\[
\mathrm{I}=(2t\kappa_1+t^2\kappa_3^2)ds^2-2dsdt,\quad \mathrm{II}=-(\kappa_2+t\kappa_3'+t^2\kappa_3^3)ds^2-2\kappa_3dsdt.
\]
Hence, the mean curvature $H$ and the Gaussian curvature $K$ are also computed as
\[
H:=\frac{1}{2}\mathrm{tr}{(\mathrm{I}^{-1}\mathrm{II})}=\kappa_3,\quad
K:=\det{({\mathrm{I}^{-1}\mathrm{II}})}=\kappa_3^2.
\]

We see that $f_L$ has a constant mean curvature $H$ if and only if $\kappa_3=-H$.
Namely, any constant mean curvature $H$ null scroll has
\begin{itemize}
\item the Frenet-Serret equations
\begin{equation}\label{eq:FS_nullscroll}
(A',B', C')=(A,B,C)\begin{pmatrix}
 \kappa_1 & 0 & H  \\
0 & - \kappa_1  &  \kappa_2\\
 \kappa_2  & H & 0
\end{pmatrix},
\end{equation}
\item the first and second fundamental forms
\[
\mathrm{I}=(2t\kappa_1+t^2H^2)ds^2-2dsdt,\quad \mathrm{II}=-(\kappa_2+t^2H^3)ds^2-2Hdsdt,
\]
\item  the Gauss map 
\begin{equation}\label{eq:N_L}
N_L=-C-tHB
\end{equation}
which satisfies $\Box N_L=2H^2N_L$, where $\Box$ is the d'Alambert operator with respect to the metric $\mathrm{I}$, see \cite{AFLM} for example. 
\end{itemize}

We call a timelike minimal surface $f$ in $\Nil(H)$ associated with a constant mean curvature $H$ null scroll (resp. B-scroll) $f_L$ in $\mathbb{L}^3$ as in Fact \ref{fact:correspond} a {\it null scroll type surface} (resp. {\it B-scroll type surface}).

In the end of this section, we recall the representation formula for B-scroll type timelike minimal surfaces in $\Nil(H)$ with respect to $(A,B,C)$ proved in \cite[Section 3]{Kiyohara}.

\begin{fact}[\cite{Kiyohara}]\label{fact:correspond_Bscroll}
Let $f_L(s,t)=\gamma(s)+tB(s)$ be a B-scroll in $\mathbb{L}^3$ with constant mean curvature and $H$ defined by \eqref{eq:Bscroll} associated with a null frame of $(A,B,C)$. Then the corresponding B-scroll type timelike minimal surface $f$ in $\Nil(H)$ in Fact \ref{fact:correspond} is written as
\begin{equation}\label{eq:correspond_Bscroll}
\begin{pmatrix} f^1\\f^2\\f^3 \end{pmatrix}=
\begin{pmatrix} \gamma_1+tB_1\\ \gamma_2+tB_2\\ \gamma_3-tB_3+H\int (-\gamma_2 \gamma'_1+\gamma_1 \gamma'_2)ds+tH(\gamma_1B_2-\gamma_2B_1) \end{pmatrix},
\end{equation}
where $\gamma=(\gamma_1,\gamma_2,\gamma_3)$ and $B=(B_1,B_2,B_3)$.
\end{fact}

\section{Singularities on null scroll type timelike minimal surfaces in $\Nil(\tau)$}\label{sec:4}
By using Theorem \ref{thm:criteria2}, we can obtain the following criteria.

\begin{theorem}\label{thm:criteria_Bscroll}
Let $f_L$ and $f$ be a null scroll with constant mean curvature $H$ in $\mathbb{L}^3$ written as \eqref{eq:Bscroll} and its corresponding null scroll type surface in $\Nil(H)$.
For a point $(s,t)\in I \times \mathbb{R}$, the following assertions hold.
\begin{itemize}
\item[(1)] $(s,t)$ is a singular point of $f$ if and only if $t= -\frac{C_3(s)}{HB_3(s)}$. Furthermore, any such a singular point $(s,t)$ is non-degenerate.
\item[(2)] $f$ is a front at a singular point $(s,t)$ if and only if $\kappa_2(s)  \neq 0$,
\item[(3)] $f$ is $\mathcal{A}$-equivalent to the cuspidal cross cap at  a singular point $(s,t)$ if and only if 
\[
\kappa_2(s)=0\quad \text{and} \quad \kappa_2'(s)=0.
\]
\end{itemize}\end{theorem}

By (1) of Theorem \ref{thm:criteria_Bscroll}, the singular set $\Sigma$ of $f$ is written as
\[
\Sigma=\{(s,t)\in I\times \mathbb{R}\mid t=-{C_3(s)}/{HB_3(s)}\}.
\]
In particular, the surface $f$ in $\Nil(H)$ has always singularities.

\begin{proof} 
Let us put $A=(A_1, A_2, A_3)$, $B=(B_1, B_2, B_3)$ and $C=(C_1, C_2, C_3)$.
By the assertion (0) in Theorem \ref{thm:criteria2}, the singular set $\Sigma$ is characterized by the equation $\langle N_L, \bm{e}_3\rangle =0$ and hence \eqref{eq:N_L} yields the relation
\begin{equation}\label{eq:singur_relation}
\langle N_L, \bm{e}_3\rangle = -C_3-tHB_3=0
\end{equation}
on $\Sigma$. Here, we can see that $B_3\neq 0$ on $\Sigma$. In fact, if we assume $B_3=0$ on a point in $\Sigma$, then $C_3=0$ holds at the same point by \eqref{eq:singur_relation} and hence there is a real constant $k$ such that
\[
\begin{pmatrix}
A_1  \\
A_2 
\end{pmatrix}
=k
\begin{pmatrix}
B_1  \\
B_2 
\end{pmatrix}
\] 
by the relations $C=A\times B$ and $C_3=0$. This with $B_3=0$ implies that two lightlike vector fileds $A$ and $B$ are parallel at this point, but it contradicts to $\langle A,B\rangle=-1$.
Therefore, by \eqref{eq:singur_relation}, the set $\Sigma$ is written as the image of the singular curve $c$ defined by
\[
c(s):=(s,t(s)) = \left(s, -\frac{C_3(s)}{HB_3(s)}\right).
\]

Next, we check non-degeneracy of singular points. By \eqref{eq:FS_Bscroll} and \eqref{eq:N_L},  $dN_L$ is computed as

\begin{align*}
dN_L&=-dC-Hd(tB)\\ 
&=-(HA+\kappa_2B)ds-H\left[Bdt +t(-\kappa_1B+HC)ds\right]\\
&=(-HA-\kappa_2B+tH\kappa_1B-tH^2C)ds-HBdt.
\end{align*}

Since $B_3\neq 0$ on $\Sigma$, then the above computation shows that $\langle dN_L,e_3\rangle \neq 0$ on $\Sigma$, which implies that any singular point on $\Sigma$ is non-degenerate by the assertion (1) in Theorem \ref{thm:criteria2}.

Next, we give a front condition. Let us consider the curve $c_L:=f_L \circ c$ on the surface $f_L$. By \eqref{eq:Bscroll} and \eqref{eq:FS_Bscroll}, the derivative $c_L'$ is computed as follows:
\begin{align} \label{eq:cL}
c_L' &= \gamma'-\left(\frac{C_3}{HB_3}\right)'B-\frac{C_3}{HB_3}B' \notag \\
&=A-\frac{1}{H}\left(\frac{C_3'B_3-C_3B_3'}{B_3^2}\right)B-\frac{C_3}{HB_3}(-\kappa_1B+HC) \notag \\
&=A-\frac{1}{H}\left[\frac{(HA_3+\kappa_2B_3)B_3-C_3(-\kappa_1B_3+HC_3)}{B_3^2}\right]B+\frac{\kappa_1}{H}\frac{C_3}{B_3}B-\frac{C_3}{B_3}C \notag \\
&=A+\left(-\frac{A_3}{B_3}-\frac{\kappa_2}{H}+\frac{C_3^2}{B_3^2} \right)B-\frac{C_3}{B_3}C.
\end{align}
Hence, we obtain 
\[
\langle c_L', \bm{e}_3\rangle =- \frac{\kappa_2B_3}{H}. 
\]
Since $B_3\neq 0$ on $\Sigma$, the surface $f$ in $\Nil(H)$ is a front at a singular point $(s, t(s))$ in $\Sigma$ if and only if $\kappa_2(s)\neq 0$.
The assertion (2) is proved.

Finally, if we assume that $f$ is not a front at a singular point $(s, t(s))$ in $\Sigma$, then $\kappa_2(s)=0$. Under this assumption,
\[
\langle c_L'', \bm{e}_3\rangle = -\frac{\kappa_2'B_3}{H}
\]
holds. By using the assertion (5) in Theorem \ref{thm:criteria2}, we have the assertion (3).
\end{proof}

%%%%%%%%%%%%%%%%%%%%%%%%%%%%%%%%%%

\subsection{Geometric meaning of $\kappa_2$ and construction of B-scroll type surfaces}

In the Frenet-Serret equations \eqref{eq:FS_nullscroll}, we may take a parameter $s$, called a {\it distinguished parameter}, such that $\kappa_1=0$ without changing $H$. See \cite[pages 58--59]{DB} for more details.

From now on, we focus on B-scrolls in $\mathbb{L}^3$ with constant mean curvature $H$ and corresponding B-scroll type minimal surfaces in $\Nil(H)$. For the null frame $(A,B,C)$ adopted to such a B-scroll, the Frenet-Serret equations \eqref{eq:FS_nullscroll} become

\begin{equation}\label{eq:FS_Bscroll}
(A',B', C')=(A,B,C)\begin{pmatrix}
0 & 0 & H  \\
0 & 0  &  \kappa_2\\
 \kappa_2  & H  & 0
\end{pmatrix}.
\end{equation}

By Theorem \ref{thm:criteria_Bscroll}, we can see that for a given smooth function $\kappa_2$ and a constant $H\neq 0$, solving \eqref{eq:FS_Bscroll} yields a B-scroll $f_L$ in $\mathbb{L}^3$ with constant mean curvature $H$ and the corresponding B-scroll type surface $f$ in $\Nil(H)$ with singularities. 
However, solving the equations \eqref{eq:FS_Bscroll} becomes complicated for a specific function $\kappa_2$. Therefore, in this section, we will describe a method for constructing such a null scroll type surface without solving the equations \eqref{eq:FS_Bscroll} and the geometrical meaning of $\kappa_2$ that is revealed by this method.

First, we see that a null frame $(A,B,C)$ satisfying \eqref{eq:FS_Bscroll} is controlled only by the lightlike vector field $B$ as follows.

\begin{lemma}\label{lem:B_construction}
Let $I\subset \mathbb{R}$ be an interval and $H$ be a non-zero constant. 
Any null frame $(A,B,C)$ defined on $I$ satisfying \eqref{eq:frame_cond} and \eqref{eq:FS_Bscroll} has the following relations.
\begin{equation}\label{eq:B_rep}
A=-\frac{\kappa_2}{H}B+\frac{1}{H^2}B'',\quad C=\frac{1}{H}B',\quad \kappa_2=-\frac{1}{2H^3}\langle B'',B''\rangle.
\end{equation}

Conversely, if $B=B(s)$ is a non-vanishing lightlike vector field defined on $I$ satisfying
\begin{equation}\label{eq:B_cond}
  \langle B', B'\rangle=H^2,\quad H\det{(B, B', B'')}>0, 
 \end{equation}
then the null frame $(A,B,C)$ defined by the relations
\begin{equation}\label{eq:ABC}
A:=-\frac{\kappa_2}{H}B+\frac{1}{H^2}B'',\quad C:=\frac{1}{H}B' 
\end{equation}
gives a solution to \eqref{eq:FS_Bscroll} with $\kappa_2:=-\dfrac{1}{2H^3}\langle B'',B''\rangle$.
\end{lemma}

\begin{remark}
In the converse part of Lemma \ref{lem:B_construction}, as we will see the following proof, the second condition of \eqref{eq:B_cond} corresponds to the condition $A\times B =C$ for $A$ and $B$ defined by \eqref{eq:ABC}. Then we can not change $B$ to $-B$.
\end{remark}

\begin{proof}
By \eqref{eq:frame_cond} and \eqref{eq:FS_Bscroll}, we have
\begin{equation}\label{eq:B'B''}
B'=HC,\quad B''=H^2A+H\kappa_2B,
\end{equation}
which gives the relations \eqref{eq:B_rep}.

Conversely, let us take a lightlike vector field $B$ satisfying \eqref{eq:B_cond} and the vector field $A, C$ defined by \eqref{eq:ABC}. We first check that $(A,B,C)$ gives a null frame. By using \eqref{eq:B_cond}, \eqref{eq:ABC} and the definition $\kappa_2:={\langle B'',B''\rangle}/{2H^3}$, we have
\begin{align*}
\langle A,B\rangle&= \left\langle -\frac{\kappa_2}{H}B+\frac{1}{H^2}B'',B\right\rangle = \frac{1}{H^2}\langle B'',B\rangle =-\frac{1}{H^2}\langle B',B'\rangle = -1,\\
\langle A,C\rangle &= \left\langle -\frac{\kappa_2}{H}B+\frac{1}{H^2}B'', \frac{1}{H}B'\right\rangle = 0,\\
\langle A,A\rangle &= \left\langle -\frac{\kappa_2}{H}B+\frac{1}{H^2}B'', -\frac{\kappa_2}{H}B+\frac{1}{H^2}B''\right\rangle =-2\frac{\kappa_2}{H^3}\langle B,B'' \rangle +\frac{1}{H^4}\langle B'',B'' \rangle\\
 &=2\frac{\kappa_2}{H^3}\langle B',B' \rangle +\frac{1}{H^4}(-2H^3\kappa_2)
 =2\frac{\kappa_2}{H^3}H^2 +\frac{1}{H^4}(-2H^3\kappa_2) =0,\\
\langle C,C\rangle&=\frac{1}{H^2}\langle B',B'\rangle =\frac{H^2}{H^2}=1.
\end{align*}
The above relations with $\langle B,B \rangle=0$ proves that $(A,B,C)$ gives a null frame. In particular, $\det{(A,B,C)} =\pm 1$ holds. Since \eqref{eq:B'B''} also holds by \eqref{eq:ABC}, 
\[
\det{(B,B',B'')}=\det{(B,HC,H^2A)}=H^3\det{(A,B,C)}.
\]
By comparing the above equation with the second relation \eqref{eq:B_cond}, we obtain $\det{(A,B,C)} =1$ and hence $\det{(B,B',B'')}=H^3$.
Therefore, we can easily check that the relation $B''\times B=H^2C$ and
\[
A\times B =\left(-\frac{\kappa_2}{H}B+\frac{1}{H^2}B''\right) \times B = \frac{1}{H^2}B''\times B = \frac{1}{H^2}(H^2C) = C.
\]

Finally, let us check $(A,B,C)$ satisfies the Frenet-Serret formulas \eqref{eq:FS_Bscroll}. 
By \eqref{eq:ABC}, 
\begin{align*}
B'&=HC,\\
C'&=\frac{1}{H}B''=\frac{1}{H}\left( H^2A+H\kappa_2B\right) =HA+\kappa_2B
\end{align*}
hold. On the other hand, $A'$ is also computed as
\begin{equation}\label{eq:A'}
A'=\left( -\frac{\kappa_2}{H}B+\frac{1}{H^2}B''\right)'
=-\frac{\kappa_2'}{H}B-\frac{\kappa_2}{H}B'+\frac{1}{H^2}B'''
=-\frac{\kappa_2'}{H}B-\kappa_2C+\frac{1}{H^2}B'''.
\end{equation}

Using \eqref{eq:B_cond} and $\kappa_2 := -\dfrac{1}{2H^3}\langle B'', B'' \rangle$, we obtain $B''' = H \kappa_2' B + 2H^2 \kappa_2 C$. Substituting this into \eqref{eq:A'}, we find that $A' = \kappa_2 C$. Hence, $(A, B, C)$ satisfies \eqref{eq:FS_Bscroll}.
\end{proof}

Next, to give a geometric meaning of $\kappa_2$, we recall the theory of null curves and the fundamental theorem of them.

%%%%%%%%%%%%

A regular curve $\gamma=\gamma(s)\colon I\rightarrow \mathbb{L}^{3}$ is called a {\it null curve} if 
	\[
		\langle \gamma', \gamma' \rangle=0, \quad  \text{$'$ denotes  $\frac{d}{ds}$}
	\]
and $\gamma$ is said to be {\it non-degenerate} if $\gamma'$ and $\gamma''$ are linearly independent at each point on $I$. 
For a non-degenerate null curve $\gamma(s)$, we can normalize the parameter so that 
	\begin{equation}\label{eq:psudo-arc}
		\langle \gamma''(s), \gamma''(s)\rangle=H^2,\quad  \text{for a non-zero constant $H$.}
	\end{equation}
The symbol $H$ is employed here so that it can later be identified with the mean curvature of null scrolls. When $H=1$,  a parameter $s$ satisfying \eqref{eq:psudo-arc} is called the {\it pseudo-arclength parameter}, see \cite{Bonnor,Vessiot} for more details.

From now on, let $s$ be a parameter satisfying \eqref{eq:psudo-arc}.
If we take the vector fields
	\[
		\vect{\sigma}(s) := \gamma'(s),\quad
		\vect{e}(s) := \gamma''(s)/H,
	\]
and then there is a unique lightlike vector field $\vect{n}$ such that
	\[
		\langle \vect{n}, \vect{\sigma} \rangle=-1,\quad
		\langle\vect{n}, \vect{e}\rangle=0.
	\]
If we set the {\it lightlike curvature} (see \cite[p.47]{IL}) of $\gamma$ to be
	\begin{equation}\label{eq:nullcurvature}
		\kappa_\gamma(s) := \inner{\vect{e}'(s)}{\vect{n}(s)},
	\end{equation}
we get
	\[
		 \vect{n}' = -\kappa_\gamma \vect{e},\qquad
		\vect{e}'= -\kappa_\gamma \vect{\sigma} +H\vect{n}.
	\]
Therefore, we obtain the following Frenet-Serret type formula for non-degenerate null curves.

\begin{proposition}[cf.~\cite{IL, Lopez}]\label{prop:Frenet}
	For a non-degenerate null curve $\gamma$ parametrized by pseudo-arclength parameter, the null frame $\mathcal{F}:=\{\vect{\sigma}, \vect{n},  \vect{e}\}$ satisfies
		\[
			\mathcal{F}' = \mathcal{F}
				\begin{pmatrix}
					0 & 0 & -\kappa_\gamma \\
					0 & 0 & H \\
					H & -\kappa_\gamma & 0
				\end{pmatrix}.
		\]
	Moreover, the lightlike curvature $\kappa_\gamma$ of $\gamma$ satisfies
		\begin{equation}\label{eq:pseudo_arc}
			2H^3\kappa_\gamma =\langle \gamma''', \gamma'''\rangle.
		\end{equation}
\end{proposition}

By Proposition \ref{prop:Frenet}, we obtain the following fundamental theorem for non-degenerate null curves in $\mathbb{L}^3$.

\begin{lemma}
Let $I\subset \mathbb{R}$ be an interval and $H$ be a non-zero constant. For a smooth function $\kappa_{\gamma}$ defined on $I$, there exists a non-degenerate null curve $\gamma$ in $\mathbb{L}^3$ such that
\[
\langle \gamma',\gamma'\rangle = 0, \quad \langle \gamma'',\gamma''\rangle = H^2,\quad \text{and}\quad \langle \gamma''',\gamma'''\rangle = 2H^3\kappa_{\gamma}
\] 
up to an isometry in $\mathbb{L}^3$.
\end{lemma} 
%%%%%%%%%%%%

Null curves also have the following Weierstrass-like representation formula as in Fact \ref{fact:correspond}.
\begin{lemma}
For any smooth regular null curve, $\gamma=(\gamma_1,\gamma_2,\gamma_3)$  satisfying $\gamma_1'\neq \gamma_2'$, there exist non-vanishing smooth functions $h$ and $\hat{\omega}$ such that
\begin{equation}\label{eq:Wformula_null}
\gamma(s) = \frac{1}{2} 
{\int \left(-1-h^2, 1-h^2, 2h\right)\hat{\omega}ds}.
\end{equation}
Moreover, when $\gamma$ is non-degenerate, we can take a parameter $s$ satisfying \eqref{eq:psudo-arc} so that $\hat{\omega}=H/h'$. In this case, the lightlike curvature $\kappa_\gamma$ is written as
\begin{equation}\label{eq:Schwarzian}
\kappa_\gamma =\frac{{S}(h)}{H}, 
\end{equation}
where $S(h)$ is the Schwarzian derivative of $h$, that is,
\[
S(h)=\frac{h'''}{h'}-\frac{3}{2}\left(\frac{h''}{h'}\right)^2.
\]

Conversely, for any smooth function $h$ with non-vanishing derivative, the curve $\gamma$ defined by \eqref{eq:Wformula_null} with $\hat{\omega}=H/h'$ is a smooth regular non-degenerate null curve whose lightlike curvature is given by \eqref{eq:Schwarzian}.
\end{lemma} 

\begin{proof}
This lemma was essentially proved by \cite{Olszak} and, in the case of holomorphic null curves in $\mathbb{C}^3$, by \cite{Pabel}. Here, for the reader's convenience, we provide only the proof of the first part.

For a null curve $\gamma=(\gamma_1,\gamma_2,\gamma_3)$, we can take functions $\hat{\omega}$ and $h$ as
\[
\hat{\omega}:=-\gamma_1'+\gamma_2',\quad h:=\frac{\gamma_3'}{-\gamma_1'+\gamma_2'},
\]
which gives the formula \eqref{eq:Wformula_null}.

When $\gamma$ is non-degenerate, $\langle \gamma'', \gamma'' \rangle>0$ holds and hence we can take a parameter $s$ of $\gamma$ satisfying \eqref{eq:psudo-arc} so that $\hat{\omega}=H/h'$. For such a parameter $s$, the derivatives of $\gamma$ can be computed as
\begin{align*}
\gamma' = \frac{H}{2h'}(-1-h^2,1-h^2,2h),\quad \gamma''= -\frac{Hh''}{2h'^2}(-1-h^2,1-h^2,2h)+H(-h,-h,1),\\
\gamma'''=-\frac{H}{2}\left(\frac{h'''h'^2-2h''^2h'}{h'^4}\right)(-1-h^2,1-h^2,2h)-\frac{Hh''}{h'}(-h,-h,1)+H(-h',-h',0).
\end{align*}
Then we have 
\begin{align}
\langle \gamma''', \gamma''' \rangle &= H^2h'\left(\frac{h'''h'^2-2h''^2h'}{h'^4}\right)+\left(\frac{Hh''}{h'}\right)^2\notag \\ 
&=2H^2\left(\frac{h'''}{h'}-3\frac{h''^2}{h'^2} \right)\notag \\
&=2H^2S(h). \label{eq:gamma'''}
\end{align}
By comparing \eqref{eq:pseudo_arc} with \eqref{eq:gamma'''}, we obtain the desired relation \eqref{eq:Schwarzian}.
\end{proof}
%%%%%%%%%%%%

\vspace{0.3cm}

In summary, if we put $\gamma$ as $\int B(s)ds$ in \eqref{eq:B_rep}, then we obtain the relation of curvatures
\begin{equation}\label{eq:curvature_relation}
-\kappa_2 = \kappa_{\gamma} = S(h)/H
\end{equation}
by combining the relations \eqref{eq:B_rep}, \eqref{eq:pseudo_arc} and  \eqref{eq:Schwarzian}.

Finally, we obtain the following construction method of timelike constant mean curvature $H$ B-scrolls and their corresponding timelike minimal surfaces in $\Nil(H)$ with singularities from given smooth functions.

\begin{theorem}\label{thm:Schwarzian_criterion}
For any non-zero constant $H$ and any smooth function $h=h(s)\colon I \to \mathbb{R}$ with non-vanishing derivative, there exists a regular timelike constant mean curvature $H$ B-scroll $f_L(s,t)=\gamma(s)+tB(s)$ in $\mathbb{L}^3$ associated with the null frame $(A,B,C)$ satisfying \eqref{eq:frame_cond}, 
\begin{equation}\label{eq:ABC_h}
A=\frac{S(h)}{H^2}B+\frac{1}{H^2}B'',\quad B= -\frac{H}{2h'} 
{ \left(-1-h^2, 1-h^2, 2h\right)},\quad C=\frac{1}{H}B'
\end{equation}
and the Frenet-Serret equations
\begin{equation}\label{eq:FS_Bscroll_2}
(A',B', C')=(A,B,C)\begin{pmatrix}
0 & 0 & H  \\
0 & 0  &  -S(h)/H \\
-S(h)/H  & H  & 0
\end{pmatrix}.
\end{equation}

Moreover, for its corresponding timelike minimal surface $f$ in $\Nil(H)$, the following assertions hold.
\begin{itemize}
\item[(1)] $(s,t)\in I\times \mathbb{R}$ is a singular point of $f$ if and only if $t=-\frac{C_3(s)}{HB_3(s)}$, where $B=(B_1,B_2,B_3)$ and $C=(C_1,C_2,C_3)$. Furthermore, any singular point is non-degenerate.
\end{itemize}

Let $(s,t)$ be a singular point of $f$. Then,
\begin{itemize}
\item[(2)] $f$ is a front at $(s,t)$ if and only if $S(h)(s)\neq 0$.
\item[(3)] $f$ is $\mathcal{A}$-equivalent to the cuspidal cross cap at $(s,t)$ if and only if $S(h)(s)= 0$ and $S(h)'(s)\neq 0$.
\end{itemize}
\end{theorem}

\begin{proof}
For given $H$ and $h=h(s)$, we can define 
\[
B(s)= \frac{H}{2h'} 
{ \left(-1-h^2, 1-h^2, 2h\right)}
\]
to be the integrand on the right-hand side of equation \eqref{eq:Wformula_null} multiplied by $-1$. Then we can easily check the relation \eqref{eq:B_cond}.
By \eqref{eq:curvature_relation}, we have $\kappa_2 =  -S(h)/H$ for the null frame $(A,B,C)$. 
Putting this $\kappa_2$ into \eqref{eq:FS_Bscroll} and \eqref{eq:B_rep}, we obtain the desired null frame $(A,B,C)$. 

The latter assertions (1), (2) and (3) are directly proved from Theorem \ref{thm:criteria_Bscroll} and \eqref{eq:curvature_relation}.
\end{proof}

\section{Invariance of singularity}
We are considering the duality between timelike constant mean curvature $H$ surfaces in $\mathbb{L}^3$ and timelike minimal surfaces in $\Nil(H)$.

In the following, for a timelike constant mean curvature $H$ surface $f_L$ in $\mathbb{L}^3$, 
we denote by $\Phi$ the map that associates the timelike minimal surface $f$ in 
$\Nil(H)$ given by Fact \ref{fact:correspond}; that is, $f = \Phi \circ f_L$.
This duality $\Phi$ is not compatible with isometries of the respective ambient spaces $\mathbb{L}^3$ and $\Nil(H)$. 
Hence, for a timelike constant mean curvature $H$ surface $f_L$ in $\mathbb{L}^3$ 
and its Lorentz transform $f_L^O := O f_L$, where $O$ is an element of the Lorentz group $O(2,1)$, the dual surface $f^O := \Phi \circ f_L^O$ is not necessarily isometric (and hence not congruent) to $f$ in $\Nil(H)$. 

\[
\begin{CD}
f_L @>\Phi>> f \\
@V O \in O(2,1) VV  \\
f_L^O @>\Phi>> f^O
\end{CD}
\vspace{0.3cm}
\]

Thus, via the duality $\Phi$, one can obtain from a single timelike minimal surface $f$ in $\Nil(H)$ the $O(2,1)$-family $\{ f^O \}_{O \in O(2,1)}$ of timelike minimal surfaces in $\Nil(H)$. 
When $f$ is of B-scroll type, each $f^O$ is also of B-scroll type.
For the null frame $(A,B,C)$ adapted to $f$, the null frame adapted to $f^O$ is given by $(OA, OB, OC)$. 
By Theorem \ref{thm:criteria_Bscroll}, a singular point $(s,t)$ of $f$ satisfies
\[
t = -\frac{C_3(s)}{H B_3(s)}.
\]
If we set $OB = (B_1^O, B_2^O, B_3^O)$ and $OC = (C_1^O, C_2^O, C_3^O)$, then the corresponding singular point $(s, t^O)$ of $f^O$ also satisfies
\[
t^O = -\frac{C_3^O(s)}{H B_3^O(s)}.
\]

Although $f$ and $f^O$ are not isometric within $\Nil(H)$ in general, it is nevertheless remarkable that the conditions (2) and (3) in Theorem~\ref{thm:criteria_Bscroll} depend only on the curvature $\kappa_2$, which is invariant under isometries of $\mathbb{L}^3$. Hence, we obtain the following invariance theorem which asserts that the properties (1) and (2) of surfaces in $\Nil(H)$ are preserved under isometries of $\mathbb{L}^3$.

\begin{corollary}[Invariance of singularities]\label{cor:invariance1}
Let $(s,t)$ be a singular point of a B-scroll type timelike minimal surface $f$ in $\Nil(H)$ and $(s,t^O)$ be the corresponding singular point of $f^O$ as above. Then the following statements hold.
\begin{itemize}
\item[(1)] $f$ is a front at $(s,t)$ if and only if $f^O$ is a front at $(s,t^O)$.
\item[(2)] $f$ is $\mathcal{A}$-equivalent to the cuspidal cross cap at $(s,t)$ if and only if $f^O$is $\mathcal{A}$-equivalent to the cuspidal cross cap at at $(s,t^O)$.
\end{itemize}
In particular, if $f$ has a cuspidal cross cap, then the family $\{f^O\}_{O\in O(2,1)}$ consists of B-scroll type timelike minimal surfaces in $\Nil(H)$ with a cuspidal cross cap.
\end{corollary}

Since such invariance properties do not hold for specific front singularities (e.g., cuspidal edges or swallowtails) of B-scroll type surfaces, these singularities cannot be characterized solely in terms of the curvature $\kappa_2$, which is an isometry invariant of $\mathbb{L}^3$. In other words, the diffeomorphism types of front singularities depend on the initial values of the null frame $(A,B,C)$ chosen when solving the Frenet--Serret equations \eqref{eq:FS_Bscroll_2}. More precisely, this can be formulated as follows.

\begin{proposition}\label{prop:NotCE}
Let $f$ be a B-scroll type timelike minimal surface associated with a null frame $(A, B, C)$ and the curvature $\kappa_2$ as in \eqref{eq:FS_Bscroll}. If $(s,t)$ be a singular point of $f$ satisfying $\kappa_2(s)\neq 0$, then following statements are equivalent.

\begin{itemize}
\item[(1)] $f$ is not $\mathcal{A}$-equivalent to the cuspidal edge at $(s,t)$,
\item[(2)] $\cfrac{\kappa_2}{H}B_3^2=1$\quad and\quad $2A_3B_3+1-C_3^2=0$ at $s$.
\end{itemize}
\end{proposition}

\begin{proof}
By Theorem \ref{thm:criteria_Bscroll} and the assumption $\kappa_2(s)\neq0$,  $f$ is a front at $(s,t)$. By Theorem \ref{thm:criteria2} and \eqref{eq:cL}, the assertion (2) is equivalent to $\langle c_L', e_i \rangle=0$ for $i=1,2$ at $s$, that is,

\begin{equation}\label{eq:A12}
A_i+\left( -\frac{A_3}{B_3}-\frac{\kappa_2}{H}+\frac{C_3^2}{B_3^2}\right)B_i-\frac{C_3}{B_3}C_i=0,\quad i=1,2,\quad \text{at $s$}.
\end{equation}

Hence, if we assume (1), the vector $A(s)$ is written as
\begin{equation}\label{eq:As}
A(s)=\left( \left( \frac{A_3}{B_3}+\frac{\kappa_2}{H}-\frac{C_3^2}{B_3^2}\right)B_1+\frac{C_3}{B_3}C_1, \left( \frac{A_3}{B_3}+\frac{\kappa_2}{H}-\frac{C_3^2}{B_3^2}\right)B_2+\frac{C_3}{B_3}C_2, A_3\right).
\end{equation}
Using the $A(s)$ as above, we have 
\begin{align*}
-1&=\langle A, B\rangle \\
&=-\left( \frac{A_3}{B_3}+\frac{\kappa_2}{H}-\frac{C_3^2}{B_3^2}\right)B_1^2-\frac{C_3}{B_3}C_1B_1+\left( \frac{A_3}{B_3}+\frac{\kappa_2}{H}-\frac{C_3^2}{B_3^2}\right)B_2^2+\frac{C_3}{B_3}C_2B_2+A_3B_3\\
&=-\left( \frac{A_3}{B_3}+\frac{\kappa_2}{H}-\frac{C_3^2}{B_3^2}\right)B_3^2-\frac{C_3}{B_3}C_3B_3+A_3B_3\\
&=-\frac{\kappa_2}{H}B_3^2,
\end{align*} 
proving the first relation of (2). Similarly, we have
\begin{align*}
0&=\langle A, A\rangle \\
&=-\left( \frac{A_3}{B_3}+\frac{\kappa_2}{H}-\frac{C_3^2}{B_3^2}\right)B_1A_1-\frac{C_3}{B_3}C_1A_1+\left( \frac{A_3}{B_3}+\frac{\kappa_2}{H}-\frac{C_3^2}{B_3^2}\right)B_2A_2+\frac{C_3}{B_3}C_2A_2+A_3^2\\
&=-\left( \frac{A_3}{B_3}+\frac{\kappa_2}{H}-\frac{C_3^2}{B_3^2}\right)(1+B_3A_3)-\frac{C_3}{B_3}C_3A_3+A_3^2\\
&=-\left( \frac{A_3}{B_3}+\frac{\kappa_2}{H}-\frac{C_3^2}{B_3^2}\right)-\frac{\kappa_2}{H}B_3A_3.
\end{align*} 
Putting the relation ${\kappa_2}/{H}=1/B_3^2$ into the above equation, we obtain
\[
0=-2 \frac{A_3}{B_3}-\frac{1}{B^2_3}+\frac{C_3^2}{B_3^2}=-\frac{1}{B_3^2}\left(2A_3B_3+1-C_3^2 \right),
\]
which proves the second relation of (2).

Conversely, let us assume (2).  
If we denote by $\tilde{A}$ the vector on the right-hand side of \eqref{eq:As}, 
then by retracing the above computation in reverse, we find that $\tilde{A}$ satisfies  
\[
\langle \tilde{A}, \tilde{A} \rangle = \langle \tilde{A}, C \rangle = 0, \quad 
\langle \tilde{A}, B \rangle = -1.
\]
This implies that $A(s)=\tilde{A}$ and hence $A(s)$ satisfies \eqref{eq:A12}, which proving (1).
\end{proof}

By (2) of Proposition~\ref{prop:NotCE}, we obtain the following relation between the curvature $\kappa_2$ and the parameter $H$ of the Heisenberg group $\Nil(H)$ for the existence and non-existence of front singularities other than cuspidal edges.

\begin{corollary}
At any singular point of a timelike minimal B-scroll type surface $f$ in $\Nil(H)$ generated by the null frame $(A,B,C)$ with $\kappa_2>0$ (resp.  $\kappa_2<0$) on which $f$ is a front, $f$ is $\mathcal{A}$-equivalent to the cuspidal edge when $H<0$ (resp. $H>0$).
\end{corollary}

\begin{proof}
The assertion holds directly from the first condition (2) of Proposition \ref{prop:NotCE}.
\end{proof}

In particular, a necessary condition for $f$ to have a singular point of a front other than the cuspidal edge can be expressed in terms of the Schwarzian derivative $S(h)$ of the function $h$ as follows.

\begin{corollary}
If a timelike minimal B-scroll type surface to $f$ in $\Nil(H)$ is a front at a singular point $(s,t(s))$ and $f$ is not $\mathcal{A}$-equivalent to the cuspidal edge, then the function $h$ in \eqref{eq:ABC_h} satisfies $S(h) < 0$.
\end{corollary}
Furthermore, under the assumption of Proposition \ref{prop:NotCE}, we obtain the following assertion, which indicates that any family $\{ f^{O} \}_{O \in O(2,1)}$ necessarily includes a timelike minimal surface in $\mathrm{Nil}(H)$ with a singular point of a front other than the cuspidal edge.

\begin{corollary}
Under the same assumption of Proposition \ref{prop:NotCE} with $\kappa_2(s)/H>0$, the family $\{f^O\}_{O\in O(2,1)}$ for each $f$ contains a timelike minimal surface with a front singularity other than the cuspidal edge.
\end{corollary}

\begin{proof}
For each $f$ with the null frame $(A,B,C)$ as in Proposition \ref{prop:NotCE}, we can always take a Lorentz transform $O\in O(2,1)$ such that $(A,B,C)$ satisfies the condition (2) of Proposition \ref{prop:NotCE} at the singular point $(s,t(s))$. Then the surface $f^O$ is a desired one. 
\end{proof}

\section{Examples}
In this section, based on the results of the previous sections,  
we shall provide concrete examples of B-scroll type timelike minimal surfaces in $ \Nil(H) $ with various singularities.
In all of the following examples, for a given function $h$, we first construct the null frame $(A, B, C)$ by applying Theorem~\ref{thm:Schwarzian_criterion}, and then obtain a B-scroll type timelike minimal surface $f$ in $\Nil(H)$ with singularities by applying Theorem~\ref{fact:correspond_Bscroll}.

First, by giving a function $h$ that satisfies condition (3) of Theorem~\ref{thm:Schwarzian_criterion}, we can construct surfaces that admit cuspidal cross caps.

\begin{example}\label{ex:CCR1}
If we take $h(s):=s+s^3$, its Schwarzian derivative $S(h)$ is computed as
\[
S(h)=\frac{6 - 36 s^2}{(1 + 3 s^2)^2}.
\]
Then it satisfies the conditions at $s^\pm :=\pm 1/\sqrt{6}$
\[
S(h)(s^\pm) = 0,\quad S(h)'(s^\pm)=\mp 16\sqrt{\frac{2}{3}}(\neq 0).
\]
Hence, by (3) of Theorem \ref{thm:Schwarzian_criterion}, the surface $f$ in $\Nil(H)$ associated with the null frame $(A,B,C)$ defined by \eqref{eq:ABC_h} for an arbitrary $H\neq 0$  is $\mathcal{A}$-equivalent to the cuspidal cross caps at $(s,t)=(s^\pm, -C_3(s^\pm)/{HB_3(s^\pm)})$. See Figure \ref{Fig:CCR} for the case $H=1$.

\begin{figure}[h!]
\vspace{-0.5cm}
\begin{center}
 \begin{tabular}{{c@{\hspace{-10mm}}c}}
\hspace{10mm}   \resizebox{8.0cm}{!}{\includegraphics[clip,scale=0.30,bb=0 0 555 449]{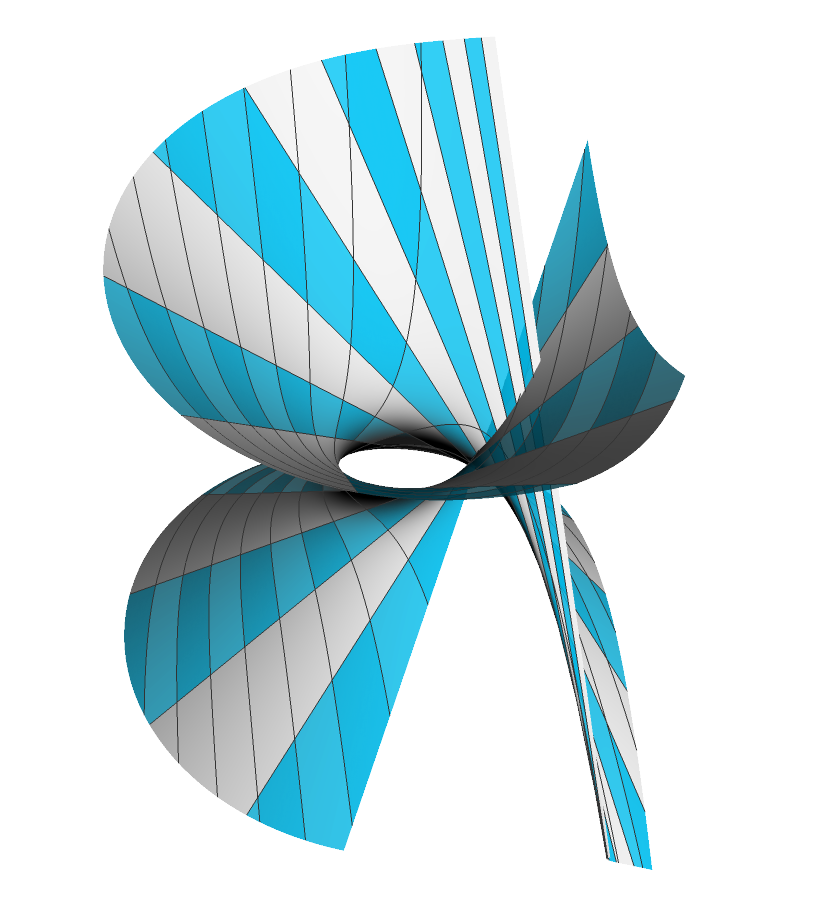}}&
 \hspace{-0mm} 
\resizebox{8.5cm}{!}{\includegraphics[clip,scale=0.30,bb=0 0 555 449]{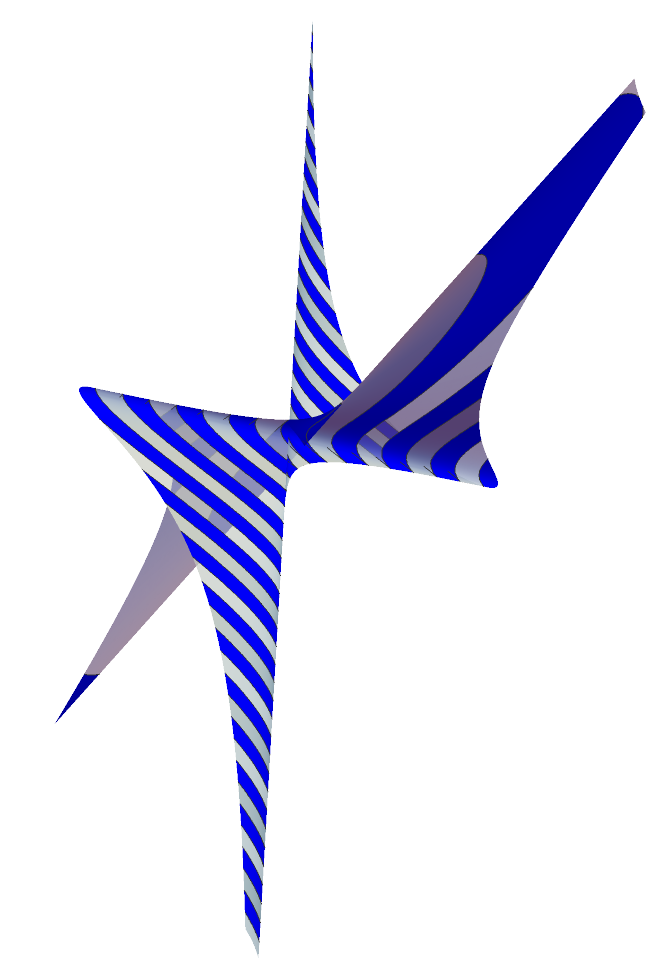}} \\
  {\hspace{-15mm}\footnotesize  timelike CMC $H=1$ B-scroll $f_L$ in $\mathbb{L}^3$} &
   {\hspace{-35mm}\footnotesize  timelike minimal B-scroll type surface $f$ in $\Nil(1)$}
 \end{tabular}
 \caption{A timelike constant mean curvature (CMC) $H=1$ B-scroll in $\mathbb{L}^3$ (left) and the corresponding timelike minimal B-scroll type surface in $\Nil(1)$ with cuspidal cross caps (right). }
 \label{Fig:CCR}
\end{center}
\end{figure}

\end{example}

\begin{example}
As another function $h$ satisfying the same conditions in Example \ref{ex:CCR1}, we may take 
$h(s) := \cot\!\left({e^s}/{2}\right)$.
The Schwarzian derivative $S(h)$ is computed as
\[
S(h)=\frac{1}{2}(e^{2s}-1),
\]
which satisfies $S(h)(0)=0$ and $S(h)'(0)=1\neq 0$.
Hence, the surface $f$ in $\Nil(H)$ associated with the null frame $(A,B,C)$ defined by \eqref{eq:ABC_h} for an arbitrary $H\neq 0$ is $\mathcal{A}$-equivalent to the cuspidal cross cap at $(s,t)=(0, -C_3(0)/{HB_3(0)})$. See Figure \ref{Fig:CCR2} for the case $H=1$.

\begin{figure}[h!]
\vspace{-1.0cm}
\begin{center}
 \begin{tabular}{{c@{\hspace{-10mm}}c}}
\hspace{10mm}   \resizebox{8.0cm}{!}{\includegraphics[clip,scale=0.30,bb=0 0 555 449]{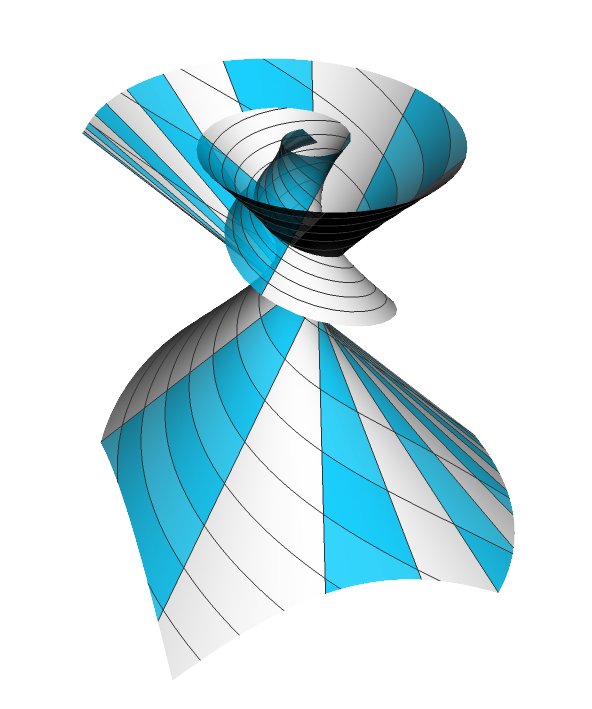}}&
 \hspace{-0mm} 
\resizebox{8.5cm}{!}{\includegraphics[clip,scale=0.30,bb=0 0 555 449]{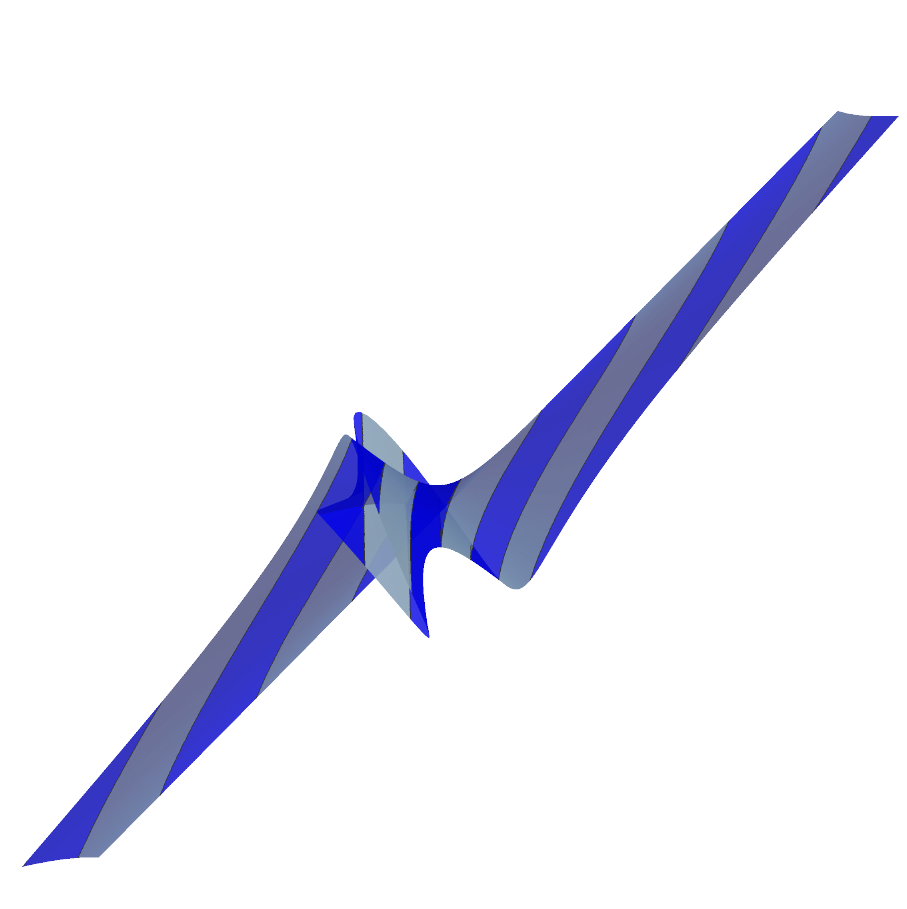}} \\
  {\hspace{-15mm}\footnotesize  timelike CMC $H=1$ B-scroll $f_L$ in $\mathbb{L}^3$} &
  {\hspace{-35mm}\footnotesize  timelike minimal B-scroll type surface $f$ in $\Nil(1)$}
 \end{tabular}
 \caption{A timelike constant mean curvature (CMC) $H=1$ B-scroll in $\mathbb{L}^3$ (left) and the corresponding timelike minimal B-scroll type surface in $\Nil(1)$ with a cuspidal cross cap (right). }
 \label{Fig:CCR2}
\end{center}
\end{figure}

\end{example}

Next, we give an example of timelike minimal B-scroll type surface with a swallowtail.

\begin{example}
For the function $h(s)=\tanh{s}$, the Schwarzian derivative is $S(h)=-2$. Then, by using $h$ and assuming $H=1$, we obtain the following null frame $(A,B,C)$ defined by \eqref{eq:ABC_h} in Theorem \ref{thm:Schwarzian_criterion}.
\[
A=(\cosh{2s}, 1, -\sinh{2s}),\quad B=\frac{1}{2}(\cosh{2s}, -1, -\sinh{2s}),\quad C=(\sinh{2s}, 0, -\cosh{2s}).
\]

The constant mean curvature $1$ B-scroll $f_L(s,t)=\gamma(s)+tB(s)$ in $\mathbb{L}^3$ associated with $(A,B,C)$ and the corresponding $f$ in $\Nil(1)$ can be computed as
\begin{align*}
f_L(s,t)&=\frac{1}{2}\left(\sinh{2s} t\cosh{2s}, 2s-t, -1-\cosh{2s}-t\sinh{2s} \right),\\
f(s,t)&=\frac{1}{2}\left(\sinh{2s} +t\cosh{2s}, 2s-t, -\frac{1}{2} -st\cosh{2s}+\left(-s+\frac{t}{2}\right)\sinh{2s}\right)
\end{align*}
Then, the curve $c_L(s):=f_L(s, -C_3(s)/B_3(s))$ in Theorem \ref{thm:criteria2} is
\[
c_L(s)=\frac{1}{2}\left( -\frac{1}{2}\left(3+\cosh{4s}\right)\mathrm{csch}{2s}, 2s+\coth{s}+\tanh{s},\sinh^2{s} \right).
\]
On the other hand, any candidate for a singular point other than the cuspidal edge must satisfy equation (2) in Proposition \ref{prop:NotCE}. Since the first equation in (2) becomes $2B_3^2 = 1$, the values of $s$ satisfying this are
\[
s^\pm:=\frac{1}{4}\log{(5\pm 2\sqrt{6})}.
\]

At this $s^\pm$, we can check that $c_L$ satisfies
\[
c_L'(s^\pm)= (0,0,\pm \sqrt{2}),\quad c_L''(s^\pm)= (\mp 6\sqrt{2},\pm 2\sqrt{6}, 2\sqrt{3}).
\]
Hence, the criterion in (4) of Theorem \ref{thm:criteria2} shows that $s = s^{\pm}$ correspond to swallowtails. See Figure \ref{Fig:SW}.
\begin{figure}[h!]
\vspace{-0.5cm}
\begin{center}
 \begin{tabular}{{c@{\hspace{-10mm}}c}}
\hspace{10mm}   \resizebox{8.0cm}{!}{\includegraphics[clip,scale=0.30,bb=0 0 555 449]{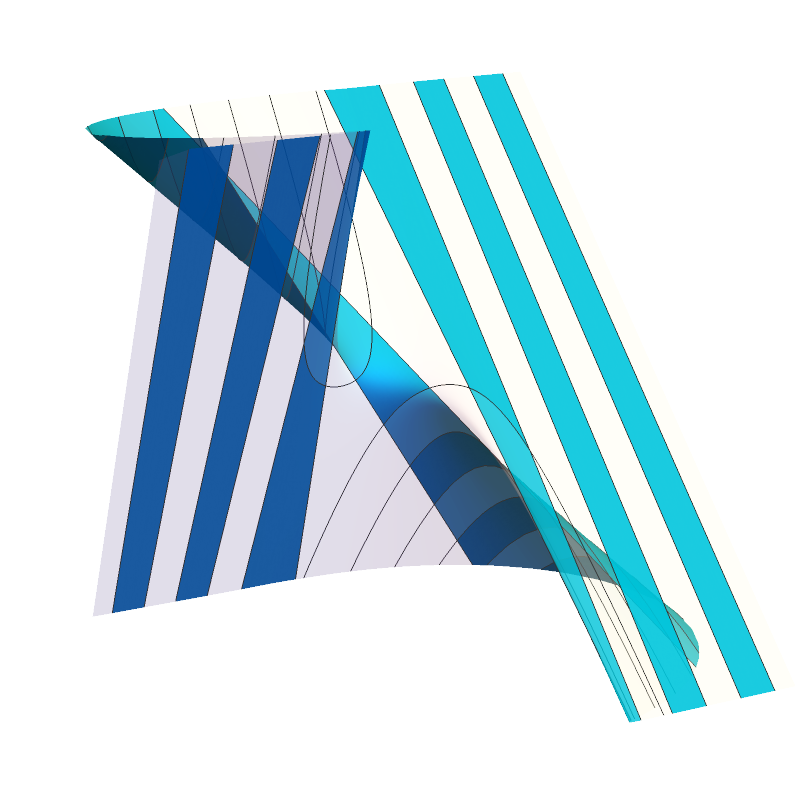}}&
 \hspace{-0mm} 
\resizebox{8.5cm}{!}{\includegraphics[clip,scale=0.30,bb=0 0 555 449]{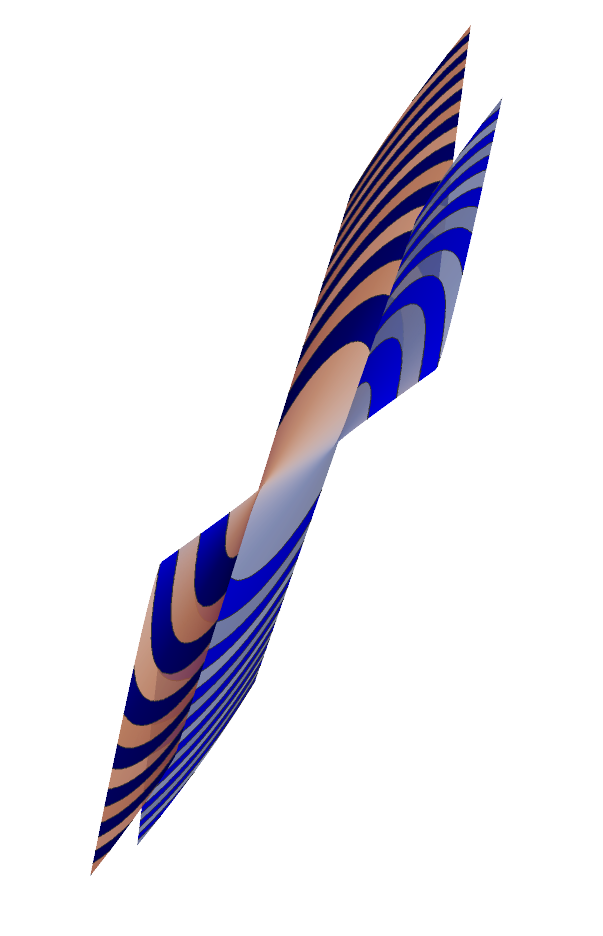}} \\
  {\hspace{-15mm}\footnotesize  timelike CMC $H=1$ B-scroll $f_L$ in $\mathbb{L}^3$} &
  {\hspace{-35mm}\footnotesize  timelike minimal B-scroll type surface $f$ in $\Nil(1)$}
 \end{tabular}
 \caption{A timelike constant mean curvature (CMC) $H=1$ B-scroll in $\mathbb{L}^3$ (left) and the corresponding timelike minimal B-scroll type surface in $\Nil(1)$ with swallowtails (right). }
 \label{Fig:SW}
\end{center}
\end{figure}

\end{example}

%%%%%%%%%%%%%%%%%%%%%%%%%%%%%%%%%%

%%%%%%%%%%%%%%%%%%%%%%%%%%%%%%%%%%

%%%%%%%%%%%%%%%

\end{document}